\documentclass[12pt]{amsart}

\newtheorem{thm}{Theorem}[section]

\newtheorem{cor}{Corollary}

\newtheorem{prp}{Proposition}

\DeclareMathOperator{\inv}{inv}

\DeclareMathOperator{\sor}{sor}
\DeclareMathOperator{\cy}{cyc}
\DeclareMathOperator{\m}{rl-min}
\DeclareMathOperator{\mb}{nmin}

\title{The sorting index}

\author[T. K. Petersen]{T. Kyle Petersen}
\address{Department of Mathematical Sciences, DePaul University, Chicago, IL, USA}
\thanks{Work supported by NSA Young Investigator grant}

\date{July 2010}

\begin{document}

\maketitle

\begin{abstract}
We consider a bivariate polynomial that generalizes both the length and reflection length generating functions in a finite Coxeter group. In seeking a combinatorial description of the coefficients, we are led to the study of a new Mahonian statistic, which we call the sorting index. The sorting index of a permutation and its type B and type D analogues have natural combinatorial descriptions which we describe in detail.
\end{abstract}

\section{Introduction}

The following is a well-known and elegant formula for the distribution of inversions over permutations of $n$: \begin{equation}\label{eq:mah}
 \sum_{w \in S_n} q^{\inv(w)} = (1+q)(1+q+q^2) \cdots ( 1+q+\cdots +q^{n-1}).
 \end{equation} (Knuth \cite[5.1.1]{K} attributes the formula to Rodrigues.)
MacMahon \cite{M} showed that another statistic, now called the \emph{major index}, has the same distribution as inversion number. In his honor, any permutation statistic with this distribution is referred to as a \emph{Mahonian statistic}.
Similarly, there is a simple formula for the distribution of the number of cycles in a permutation (see, e.g., \cite[Prop. 1.3.4]{St}): 
\begin{equation}\label{eq:cyc}
  \sum_{w \in S_n} t^{\cy(w)} = t(t+1)(t+2)\cdots (t+n-1). 
\end{equation}
The coefficients of this polynomial are well-known as the (unsigned) \emph{Stirling numbers of the first kind}.

Both Equations \eqref{eq:mah} and \eqref{eq:cyc} can be given simple combinatorial proofs, and it is interesting that they both factor with the same number of terms. This suggests a natural bivariate polynomial: \[ S_n(q,t):= t(t+q)(t+q+q^2)\cdots (t+q+\cdots+q^{n-1}),\] which clearly generalizes both polynomials above. This polynomial does not give the joint distribution of inversion number and cycle number. However, we find statistics $\sor(w)$ and $\m(w)$ such that $S_n(q,t)$ is the generating function for both the pair of statistics $(\inv,\m)$ and the pair $(\sor,\cy)$. In particular, $\sor(w)$ is Mahonian, and $\m(w)$ has a ``Stirling" distribution. While a variation of $\m(w)$ appears early in \cite{St}, $\sor(w)$ seems not to have appeared in the literature. Given the variety of distinct Mahonian statistics (see \cite{BS}, \cite{Cl}, \cite{CSZ}, \cite{FZ}, \cite{FZ2}, \cite[A008302]{Sl} and references therein) we find the novelty of $\sor(w)$ noteworthy. We call $\sor(w)$ the \emph{sorting index} of $w$.

In Section \ref{sec:A} we make precise the story told above, and we follow that discussion with very similar results for signed permutations in Section \ref{sec:B}.  In the context of reflection groups, $\inv(w)$ corresponds to \emph{length} and $n-\cy(w)$ corresponds to \emph{reflection length}.  Sections \ref{sec:D} and \ref{sec:W} discuss how the ideas of Sections \ref{sec:A} and \ref{sec:B} might be generalized to other Coxeter groups. Indeed, the discovery of the sorting index followed from the investigation of a common refinement of the length and reflection length generating functions in an arbitrary finite reflection group. The polynomial $S_n(q,t)$ given above is the type A version. See Section \ref{sec:W}. Another avenue for generalization could be from permutations to \emph{words} (with repeated letters) as in \cite[Section III]{M}, though we do not pursue this line of inquiry here. In Section \ref{sec:rem} we remark upon recent work of others related to the type A version of the sorting index.

Throughout the paper, we will let $[n]_q=\frac{1-q^n}{1-q} = 1+q + \cdots +q^{n-1}$. We refer the reader to \cite[Sections 8.1 and 8.2]{BB} for facts about the combinatorics of Coxeter groups (in particular combinatorial descriptions for inversions of types $B_n$ and $D_n$) that we use here without proof.

\section{Permutations (type A)}\label{sec:A}

One generating set for $S_n$ is the set of transpositions:
\[ T = \{ (ij) : 1 \leq i < j \leq n \}; \]
another is the set of adjacent transpositions:
\[ S = \{ (i \, i+1) : 1\leq i \leq n-1\} \subset T.\] We will usually write elements $w \in S_n$ as words on the set $\{1,2,\ldots,n\}$ or as products of transpositions in $S$ or $T$.  We will abbreviate transpositions by $t_{ij} = (ij)$ and adjacent transpositions by $s_i = (i \, i+1)$.

In what follows, we will give two different factorizations of the \emph{diagonal sum}, i.e., $\sum_{w \in S_n} w$, in the group algebra $\mathbb{Z}[S_n]$. We then examine how two pairs of statistics, $(\inv,\m)$ and $(\sor,\cy)$, behave under the factorizations. Because both factorizations lead to the expression \[ S_n(q,t) := \prod_{i=1}^{n} (t+[i]_q-1),\] we will conclude that the joint distribution of the pairs of the statistics are the same. We refer $\inv$ and $\sor$ as \emph{Mahonian statistics}, $\m$ and $\cy$ as \emph{Stirling statistics}.

\subsection{Mahonian statistics}\label{sec:mahA}

Given a permutation $w \in S_n$, the minimal number of adjacent transpositions needed to express $w$ is called its \emph{length}, $\ell(w)$. An \emph{inversion} of $w$ is a pair $i < j$ such that $w_i > w_j$. Let $\inv(w)$ denote the number of inversions of $w$. It is well-known that the number of inversions in a permutation is equal to its length: $\inv(w) = \ell(w)$. For example, $\ell(312) = 2$ since $312=s_2s_1$, but also $312$ has two inversions since $3>1$ and $3>2$. Any statistic with the same distribution as inversion number is called \emph{Mahonian}.

\begin{table}
\begin{tabular}{ c | c | c}
 & $w = 6372451$ & $w = 3715246$\\
 \hline
  & & \\
 $ \inv(w) =\ell(w) $ & 14 & 9\\
  & & \\
  & & \\
 $\sor(w)$ & 18 & 14 \\
  & & \\
 \hline
  & & \\
 $\cy(w) = n-\ell'(w)$ & 1 & 2\\
  & & \\
 $\m(w)$ & 1 & 4\\
  & & \\
\end{tabular}
\caption{Examples of inversion number versus sorting index, and $\cy(w)$ versus $\m(w)$.}\label{tab:stat}
\end{table}

We now define the \emph{sorting index} of a permutation.  Given a permutation $w \in S_n$, there is a unique expression \[w = t_{i_1 j_1}\cdots t_{i_k j_k}\]  as a product of transpositions with $j_1 <\cdots < j_k$. The sorting index, $\sor(w)$, is given by \[ \sor(w) = \sum_{s=1}^k j_s - i_s.\] The transpositions in this factorization are precisely the transpositions used in the ``straight selection sort" algorithm, which we now describe. Using a transposition we first move the largest number that is out of order to the end, then we move the next largest number to its proper place, and so on. (See Knuth \cite[5.2.3]{K}.) For example, if $w = 2431756$, we have \[ 2431\mathbf{7}56 \xrightarrow{(57)} 2431\mathbf{6}57 \xrightarrow{(56)} 2\mathbf{4}31567 \xrightarrow{(24)} \mathbf{2}134567 \xrightarrow{(12)} 1234567, \] and so \[ w = t_{12}t_{24}t_{56}t_{57}, \qquad \sor(w) = (2-1)+(4-2) + (6-5) + (7-5) = 6.\] The sorting index, then, is the sum of the distances that all the letters had to move during the sorting. We will show that this statistic has the same distribution as inversion number.

\subsection{Stirling statistics}

For $w \in S_n$, the \emph{reflection length} of $w$, $\ell'(w)$, is the minimal number of reflections (in $T$) needed to express $w$. Let $\cy(w)$ denote the number of cycles of $w$. Reflection length and cycle number are related by $\cy(w) = n-\ell'(w)$. 

For $w \in S_n$ define the following statistic:
\[ \m(w) =  n-|\{ i : w_i > w_j \mbox{ for some } j > i\}| = |\{ i : w_i < w_j \mbox{ for all } i < j\}|,\] that is, $\m(w)$ is the number of letters of $w$ that are right-to-left minima. For example, $w = 2413756$ has $\m(w) =4$.  We will see that $\m$ has the same distribution as $\cy$.

\subsection{Factorizations of the diagonal sum}\label{sec:Afac}

One of the simplest ways to generate all permutations of $n$ is defined recursively. For each permutation of $n-1$, simply insert the letter $n$ in all possible positions. We now describe a way to encode this idea algebraically. Define elements $\Psi_i$ of the group algebra of $S_n$ as: $\Psi_1 = 1 + s_1$, $\Psi_2 = 1 + s_2 + s_2s_1 = 1+ s_2\cdot\Psi_1$, and in general for $i\geq 1$, $\Psi_i =1 + s_i\cdot \Psi_{i-1}$. It should be clear that each term in $\Psi_i$ is a reduced expression in $S_{i+1}$. The following proposition shows how the product of these terms, $1\leq i < n$, produces reduced expressions for all elements of $S_n$.

\begin{prp}\label{prp:prodA}
For $n\geq 2$ we have \[ \Psi_{1}\Psi_{2}\cdots \Psi_{n-1}= \sum_{w\in S_n} w.\]
\end{prp}

\begin{proof}
Clearly $\Psi_1 = \sum_{w\in S_{2}} w$. For induction, suppose $\Psi_1\cdots \Psi_{n-2}= \sum_{w\in S_{n-1}} w$. We can identify the elements of $S_{n-1}$ with the set \[ \{ w \in S_n :  w(n) = n\}.\] (In fact, written as products of the simple transpositions above, they are identical.) Given such an element $w = w_1 \cdots w_{n-1}n$, the terms in $\Psi_{n-1}$ dictate the location of $n$:
\begin{align*}
w\cdot \Psi_{n-1} &= w + w s_{n-1} + w s_{n-1}s_{n-2} + \cdots w s_{n-1} s_{n-2}\cdots s_1 \\
&= w_1 \cdots w_{n-1}n + w_1 \cdots w_{n-2}nw_{n-1} + \cdots + nw_1\cdots w_{n-1}.
\end{align*}
It is clear moreover that $w \cdot\Psi_{n-1} = v\cdot \Psi_{n-1}$ if and only if $w = v$. We have \[ \Psi_1\cdots \Psi_{n-2}\Psi_{n-1}= \sum_{w\in S_n, w(n) = n} w\cdot\Psi_{n-1} = \sum_{u \in S_n} u,\] as desired.
\end{proof}

We can now state the following corollary to Proposition \ref{prp:prodA}. For $t=1$, this is Equation \eqref{eq:mah}. 

\begin{cor}\label{cor:Ainv}
For $n\geq 1$ we have
\[S_n(q,t) = \prod_{i=1}^{n} (t + [i]_q-1) = \sum_{w \in S_n} q^{\inv(w)}t^{\m(w)}.\]
\end{cor}

This expression for the generating function of $(\inv,\m)$ can be seen as a special case of Theorem 1.2 of \cite{FH2005} (see Remark 2, page 17).

\begin{proof}
We find it convenient to prove an equivalent factorization, namely, \[t^nS_n(q,1/t) = \prod_{i=1}^{n} (1+t[i]_q - t) = \sum_{w \in S_n} q^{\inv(w)}t^{n-\m(w)}.\]

Define the linear map $\psi: \mathbb{Z}[S_n] \to \mathbb{Z}[q,t]$ given by $\psi(w) = q^{\inv(w)}t^{n-\m(w)}$. By construction, $\psi(\Psi_i) = (1 + qt + q^2t + \cdots + q^it) = (1+t[i+1]_q - t)$. Suppose for induction that \[\psi(\Psi_1\cdots \Psi_{n-2}) = \psi(\Psi_1)\cdots \psi(\Psi_{n-2}) = \sum_{w \in S_{n-1}} q^{\inv(w)}t^{\m(w)}.\] It suffices to show that the following holds: \[ \psi(\Psi_1\cdots \Psi_{n-2})\psi(\Psi_{n-1}) = \psi(\Psi_1\cdots\Psi_{n-2}\Psi_{n-1}).\] 

Following the proof of Proposition \ref{prp:prodA}, we see that $\Psi_{n-1}$ has the effect of positioning $n$ in a permutation. If $n$ is placed in the rightmost position it participates in no inversions and is a right-to-left minimum. If $n$ is inserted so that there are $i>0$ letters to its right, then because it is the greatest letter it can never be a right-to-left minimum, and it participates in $i$ inversions. Thus it contributes weight $q^it$. For a permutation $w = w_1 \cdots w_{n-1} n$, we have:
\begin{align*}
\psi(w\cdot \Psi_{n-1}) &= \psi( w_1 \cdots w_{n-1}n) + \psi( w_1 \cdots w_{n-2}nw_{n-1}) + \cdots + \psi(nw_1\cdots w_{n-1})\\
&=\psi(w) + qt\psi(w) + \cdots + q^{n-1} t\psi(w)\\
&=\psi(w)(1 + t[n]_q -t) = \psi(w)\psi(\Psi_{n-1}).
\end{align*}
Thus, 
\begin{align*}
\psi(\Psi_1\cdots\Psi_{n-2}\Psi_{n-1}) &= \psi\left(\sum_{w \in S_n, w(n) = n} w\cdot\Psi_{n-1} \right) \\
&= \sum_{w \in S_n, w(n) = n} \psi(w\cdot\Psi_{n-1}) \\
&= \psi(\Psi_{n-1})\sum_{w \in S_n, w(n) = n} \psi(w)\\
&=\psi(\Psi_{n-1})\psi(\Psi_1\cdots \Psi_{n-2}),
\end{align*}
as desired.
\end{proof}

We now present a different factorization of the diagonal sum. Intuitively, the procedure it describes is: given a permutation of $n-1$, put letter $n$ in any position $i = 1,\ldots, n-1$, and move the letter that was in position $i$ to the end. Let $\Phi_1 = 1 + t_{12}$, $\Phi_2 = 1 + t_{23}+ t_{13}$, and for $j \geq 2$ let $\Phi_j = 1+\sum_{i\leq j} t_{i\, j+1}$. The following proposition shows that the product of these terms produces a minimal length product of reflections for each element of $S_n$.

\begin{prp}\label{prp:prodA2}
For $n\geq 2$, we have \[ \Phi_1\Phi_2\cdots \Phi_{n-1}= \sum_{w\in S_n} w.\]
\end{prp}

\begin{proof}
Obviously the formula holds for $n=2$. For induction, suppose $\Phi_1\cdots \Phi_{n-2} = \sum_{w \in S_{n-1}} w$. We observe that $\Phi_{n-1}$ is  $1$ plus the sum of all reflections involving the letter $n$. Thus, for $w = w_1 \cdots w_{n-1} n$ we have
\begin{align*}
w\cdot \Phi_{n-1} &= w + w t_{n-1\,n}+ wt_{n-2\,n}  + \cdots + wt_{1 n} \\
&= w_1 \cdots w_{n-1} n+ w_1 \cdots n w_{n-1} + w_1 \cdots n w_{n-1} w_{n-2} + \cdots + n w_2 \cdots w_{n-1} w_1.
\end{align*}
From these descriptions it is clear that $w\cdot \Phi_{n-1} = v\cdot\Phi_{n-1}$ if and only if $w=v$, and so we have the desired result: \[ \Phi_1\cdots\Phi_{n-2}\Phi_{n-1} = \sum_{w \in S_n, w(n) = n} w\cdot\Phi_{n-1} = \sum_{w \in S_n} w.\]
\end{proof}

By looking at the combinatorics of the factorization in Proposition \ref{prp:prodA2}, we have the following.

\begin{cor}
For $n\geq 1$, we have \[S_n(q,t) = \prod_{i=1}^n (t+[i]_q - 1) = \sum_{w \in S_n} q^{\sor(w)}t^{\cy(w)}.\] 
\end{cor}

\begin{proof}
We find it convenient to prove an equivalent factorization, namely, \[t^nS_n(q,1/t) = \prod_{i=1}^{n} (1+t[i]_q - t) = \sum_{w \in S_n} q^{\sor(w)}t^{n-\cy(w)} = \sum_{w \in S_n} q^{\sor(w)}t^{\ell'(w)}.\]

Define the linear map $\phi: \mathbb{Z}[S_n] \to \mathbb{Z}[q,t]$ given by $\phi(w) = q^{\sor(w)}t^{\ell'(w)}$. It is clear that $\phi(\Phi_i) = (1 + qt + q^2t + \cdots + q^it) = (1+t[i+1]_q -t)$. It suffices to show \[ \phi(\Phi_{1}\cdots \Phi_{n-2}) \phi(\Phi_{n-1})= \phi(\Phi_1\cdots\Phi_{n-2}\Phi_{n-1}).\] Following the proof of Proposition \ref{prp:prodA2} we let $w = w_1 \cdots w_{n-1} n$. Notice that since $w_n = n$, then for $1\leq i < n$, $\ell'(w t_{in}) = \ell'(w) + 1$. Hence we have:
\begin{align*}
\phi(w\cdot\Phi_{n-1}) &= \phi(w) + \phi(w t_{n-1\, n}) + \cdots + \phi(w t_{1n})\\
&=\phi(w) + qt\phi(w) + \cdots + q^{n-1} t\phi(w)\\
&=\phi(w)(1 + t[n]_q -t) = \phi(w)\phi(\Phi_{n-1}).
\end{align*}
Thus, 
\begin{align*}
\phi(\Phi_1\cdots \Phi_{n-2}\Phi_{n-1}) &= \phi\left(\sum_{w \in S_n, w(n) = n} w\cdot\Phi_{n-1} \right) \\
&= \sum_{w \in S_n, w(n) = n} \phi(w\cdot \Phi_{n-1}) \\
&= \phi(\Phi_{n-1})\sum_{w \in S_n, w(n) = n} \phi(w)\\
&=\phi(\Phi_{n-1})\phi(\Phi_1\cdots \Phi_{n-2}),
\end{align*}
as desired.
\end{proof}

Taken together, these corollaries give our first main result.

\begin{cor}\label{cor:Adis}
The pairs of statistics $(\inv, \m)$ and $(\sor,\cy)$ are equidistributed over $S_n$:
\[ \sum_{w \in S_n} q^{\inv(w)}t^{\m(w)} = \sum_{w \in S_n} q^{\sor(w)} t^{\cy(w)}.\] In particular, $\sor$ is a Mahonian statistic, and $\m$ is a Stirling statistic.
\end{cor}

We remark that there is a bijection that maps a permutation $w$ with $k$ cycles to a permutation $\hat{w}$ with $k$ right-to-left minima, and thus $\cy(w) = \m(\hat{w})$. (This is a variation of an idea given by Stanley \cite[Proposition 1.3.1]{St}.) We will illustrate the bijection with an example. First, we write $w$ in cycle notation, so that each cycle has its least element written last, and the cycles are in increasing order by least element. With $w = 685942317$, we write it as $w = (6281)(54973)$.  We simply remove parentheses to achieve the one-line notation for $\hat{w}$: \[ \hat{w} = 628154973.\] To reverse the process, simply insert ``)" to the right of each right-to-left minimum, and put ``(" at the far left and following any internal right parentheses. One can check with this example that $\cy(w) = 2 = \m(\hat{w})$.

While the map $w \mapsto \hat{w}$ shows that $\cy$ and $\m$ are equidistributed, it does not carry $\sor(w)$ to $\inv(w)$ (in the example above $\sor(w) = 21$ and $\inv(\hat{w}) = 17$), and so it does not give a bijective proof of Corollary \ref{cor:Adis}. It would be interesting to find a bijective proof of Corollary \ref{cor:Adis}, or to find a simple bijection that carries $\inv$ to $\sor$. While one can define such a map inductively via our factorizations, a more ``natural" description would be better.

\section{Type B}\label{sec:B}

The hyperoctahedral group $B_n$ is the set of permutations of $\{ 1, \ldots, n, \bar 1, \ldots, \bar n\}$ (where $\bar i = -i$) that are centrally symmetric, i.e., elements \[w = w_{\bar n} \cdots w_{\bar 1} w_1 \cdots w_n\] for which $w(- i) = -w(i)$. Thus $w$ is determined by  the word $w =w_1\cdots w_n$. That is, elements of $B_n$ are \emph{signed permutations}. These may be generated by the following transpositions: \[ T^B = \{ (i j) : 1\leq i < j \leq n \} \cup \{ (\bar i j) : 1 \leq i \leq j \leq n \},\] where, in order to maintain symmetry, the transposition $(ij)$  means to swap both $i$ with $j$ and $\bar i$ with $\bar j$ (provided $i\neq \bar j$). The simple transpositions we denote \[ S^B = \{ (\bar 1 1) \} \cup \{ (i \, i+1): 1\leq i \leq n-1 \},\] and these form a minimal generating set for $B_n$. As before, we abbreviate the transpositions by $t_{i j} = (i j)$ and $s_i = (i \, i+1)$, with $s_0 = (\bar 1 1)$. 

\subsection{Mahonian statistics}

For a signed permutation $w\in B_n$, the minimal number of terms in $S^B$ needed to express $w$ is called its \emph{length}, $\ell_B(w)$. As with permutations, the length of an element in $B_n$ (as a reflection group) is equal to its inversion number: $\ell_B(w) = \inv_B(w)$, which we now define. Let $N(w)$ denote the number of bars in $w_1\cdots w_n$, and define the type $B_n$ inversion number as follows:
\[ \inv_B(w) = | \{ 1\leq i < j \leq n : w(i) > w(j)\}| + |\{1\leq i < j\leq n : -w(i) > w(j) \}| + N(w).\] For example, $\inv_B( 2\bar 4 5 \bar 1\bar 3) = (3+2+1) + (2+2+1) + 3 = 14$.  By analogy with the type A case, we call any statistic with the same distribution as length a Mahonian statistic. 

We will define a ``type B" straight selection sort algorithm below. The algorithm achieves a unique factorization as a product of signed transpositions: \[ w = t_{i_1 j_1} \cdots t_{i_k j_k},\] with $0<j_1 < \cdots < j_k$. The \emph{type $B_n$ sorting index}, $\sor_B(w)$, is defined to be \[ \sor_B(w) = \sum_{s=1}^k j_s - i_s - \chi(i_s < 0).\] For example, if $w = 2 \bar 4 5 \bar 1 \bar 3 $, we have \[ 3 1 \bar 5 4 \bar 2\, 2 \bar 4 \mathbf{5} \bar 1 \bar 3 \xrightarrow{(35)} \bar 5 1 3 \mathbf{4} \bar 2\, 2 \bar 4 \bar 3 \bar 1 5 \xrightarrow{(\bar 2 4)} \bar 5 \bar 4 \mathbf{3} \bar 1 \bar 2\, 2 1 \bar 3 4 5 \xrightarrow{(\bar 3 3)} \bar 5 \bar 4 \bar 3 \bar 1 \bar 2\, \mathbf{2} 1 3 4 5 \xrightarrow{(12)} \bar 5 \bar 4 \bar 3 \bar 2 \bar 1 \, 1 2 3 4 5 \] and so: \[w = t_{12}t_{\bar 3 3} t_{\bar 2 4} t_{35} \qquad \sor_B(w) = (2-1)+ (3-(-3)-1) + (4-(-2)-1) + (5-3) = 13.\]
As before, the sorting index in some sense keeps track of the distances traveled by the elements being sorted. We will show that the type B sorting index is Mahonian over $B_n$.

\subsection{Stirling statistics}

The minimal number of transpositions in $T^B$ needed to express $w$ is the \emph{reflection length}, $\ell'_B(w)$. We know of no simple way to compute the reflection length of an element of $B_n$ (apart from sorting). It would be nice to find a formula for reflection length involving, say, the cycle structure of $w \in B_n$ (perhaps viewed as an element of $S_{2n}$). Nevertheless, the coefficients of the polynomial \[ \sum_{w \in B_n} t^{\ell'_B(w)} =(1+t)(1+3t)\cdots(1+(2n-1)t) = \sum_{k=1}^n c(n,k)t^{n-k},\] are known as the ``type B" Stirling numbers of the first kind.

Define another statistic,
\[ \mb(w) = | \{ i : w_i > |w_j| \mbox{ for some } j > i\}| + N(w).\] For example, if $w = 3 4 \bar 1 8 7\bar 6 2 5$, then $\mb(w) = 4+2 =6$, while one can check $\ell'_B(w) = 7$. We will show $\mb$ is equidistributed with reflection length.
 
\subsection{Factorizations of the diagonal sum}

Just as we can generate all of $S_n$ inductively by inserting the letter $n$ into a permutation of $n-1$ in all possible positions, so we can generate $B_n$ from $B_{n-1}$ by inserting $n$ or $\bar n$ in all possible ways. Similarly to Section \ref{sec:Afac}, define elements $\Psi_i$ of the group algebra of $B_n$ as follows: $\Psi_1 = 1 + s_0$, $\Psi_2 = 1 + s_1 + s_1s_0 + s_1 s_0 s_1 = 1 + s_1\cdot\Psi_1 + s_1 s_0 s_1$, and in general for $i> 1$ define $\Psi_i = 1 +  s_{i-1}\cdot\Psi_{i-1} + s_{i-1}\cdots s_1 s_0 s_1 \cdots s_{i-1}$.

\begin{prp}\label{prp:prodB}
We have \[ \Psi_{1}\Psi_{2}\cdots \Psi_n= \sum_{w\in B_n} w.\]
\end{prp}

\begin{proof}
Clearly $\Psi_1 = \sum_{w\in B_1} w$. For induction, suppose $\Psi_{1}\cdots \Psi_{n-1}= \sum_{w\in B_{n-1}} w$. We can identify the elements of $B_{n-1}$ with the set \[ \{ w \in B_n :  w(n) = n\}.\] Given such an element $w = w_1 \cdots w_{n-1}n$, the terms in $\Psi_n$ dictate the location and sign of $n$:
\begin{align*}
w\cdot\Psi_n &= w + ws_{n-1} + ws_{n-1}s_{n-2} + \cdots + w(s_{n-1}\cdots s_{1}) \\
& \quad + w(s_{n-1}\cdots s_{1}s_0) + w(s_{n-1}\cdots s_1s_0s_1) + \cdots + w(s_{n-1}\cdots s_1 s_0 s_1\cdots s_{n-1}) \\
&= w_1 \cdots w_{n-1}n + w_1 \cdots w_{n-2}nw_{n-1} + w_1\cdots nw_{n-2}w_{n-1} + \cdots + nw_1\cdots w_{n-1} \\
& \quad +\bar n w_1 \cdots w_{n-1} + w_1\bar n w_2 \cdots w_{n-1} + \cdots + w_1\cdots w_{n-1}\bar n.
\end{align*}
Moreover, from this description it is clear that $w\cdot\Psi_{n} = v\cdot\Psi_{n}$ if and only if $w = v$. We have \[ \Psi_{1}\cdots\Psi_{n-1}\Psi_n= \sum_{w\in B_n, w(n) = n} w\cdot\Psi_{n} = \sum_{u \in B_n} u,\] as desired.
\end{proof}

We will now describe how this product formula relates to signed permutation statistics.

\begin{cor}
For $n\geq 1$, we have \[ B_n(q,t) := \prod_{i=1}^n (1+ t[2i]_q -t) = \sum_{w \in B_n} q^{\inv_B(w)}t^{\mb(w)}.\]
\end{cor}

As with Corollary \ref{cor:Ainv}, this generating function for $(\inv_B, \mb)$ is obtainable from \cite[Theorem 1.2]{FH2005}.

\begin{proof}
Define the linear map $\psi: \mathbb{Z}[B_n] \to \mathbb{Z}[q,t]$ given by $\psi(w) = q^{\inv_B(w)}t^{\mb(w)}$. It is clear that $\psi(\Psi_i) = (1 + qt + q^2t + \cdots + q^{2i-1}t) = (1+ t[2i]_q -t)$. Assume for induction that \[\psi(\Psi_{1}\cdots \Psi_{n-1}) = \psi(\Psi_{1})\cdots \psi(\Psi_{n-1}) = \sum_{w \in B_{n-1}} \psi(w).\] It suffices to show that the following holds: \[ \psi(\Psi_{1}\cdots \Psi_{n-1})\psi(\Psi_n) = \psi(\Psi_{1}\cdots \Psi_{n-1}\Psi_n).\] 

Following the proof of Proposition \ref{prp:prodB}, we see that $\Psi_n$ has the effect of positioning $n$ in a signed permutation and determining its sign. Following the characterizations of $\mb$ and $\inv_B$ above, we have,  for a signed permutation $w = w_1 \cdots w_{n-1} n$:
\begin{align*}
\psi(w\cdot \Psi_n) &= \psi( w_1 \cdots w_{n-1}n) + \psi( w_1 \cdots w_{n-2}nw_{n-1}) + \cdots + \psi(nw_1\cdots w_{n-1})\\
& \quad + \psi(\bar n w_1\cdots w_{n-1}) + \psi(w_1 \bar n \cdots w_{n-1}) + \cdots +\psi(w_1 \cdots w_{n-1}\bar n),\\
&=\psi(w) + qt\psi(w) + \cdots + q^{n-1} t\psi(w) \\
&\quad + q^n t\psi(w) + q^{n+1}t \psi(w)+ \cdots + q^{2n-1}t\psi(w),\\
&=(1 + t[2n]_q -t)\psi(w) = \psi(\Psi_n)\psi(w).
\end{align*}
Thus, 
\begin{align*}
\psi(\Psi_1\cdots\Psi_{n-1}\Psi_{n}) &= \psi\left(\sum_{w \in B_n, w(n) = n} w\cdot\Psi_n \right) \\
&= \sum_{w \in B_n, w(n) = n} \psi(w\cdot\Psi_n) \\
&= \psi(\Psi_n)\sum_{w \in B_n, w(n) = n} \psi(w)\\
&=\psi(\Psi_n)\psi(\Psi_{1}\cdots \Psi_{n-1}),
\end{align*}
as desired.
\end{proof}

We now present a different factorization of the sum $\sum_{w \in B_n} w$. Let $\Phi_1 = 1 + t_{\bar 1 1}$, $\Phi_2 = 1 + t_{12}+ t_{\bar 1 2} + t_{\bar 2 2}$, and for $j \geq 2$ let $\Phi_j = 1+\displaystyle\sum_{ j > i \geq \bar j} t_{i\, j}$.

\begin{prp}\label{prp:prodB2}
For $n\geq 1$, we have \[ \Phi_{1}\Phi_{2}\cdots \Phi_n= \sum_{w\in B_n} w.\]
\end{prp}

\begin{proof}
Obviously the formula holds for $n=1$. For induction, suppose $\Phi_1\cdots \Phi_{n-1} = \sum_{w \in B_{n-1}} w$. We observe that $\Phi_{n}$ is  $1$ plus the sum of all transpositions involving the letter $n$. Thus, for $w = w_1 \cdots w_{n-1} n$ we have
\begin{align*}
w\cdot\Phi_{n} &= w + wt_{n-1\,n}  + wt_{n-2\,n} + \cdots + wt_{1 n} + wt_{\bar 1 n}+ wt_{\bar 2 n} + \cdots + wt_{\bar n n} \\
&= w_1 \cdots w_{n-1} n+ w_1 \cdots n w_{n-1} + w_1 \cdots n w_{n-1} w_{n-2} + \cdots + n w_2 \cdots w_{n-1} w_1 \\
& \quad + \bar n w_2 \cdots w_{n-1} \bar{w}_1 + w_1 \bar n \cdots w_{n-1} \bar{w}_2 + \cdots + w_1 \cdots w_{n-1} \bar n.
\end{align*}
Clearly $w\cdot\Phi_{n} = v\cdot\Phi_{n}$ if and only if $w=v$, and so we have the desired result: \[ \Phi_{1}\cdots\Phi_{n-1} \Phi_n = \sum_{w \in B_n, w(n) = n} w\cdot\Phi_{n} = \sum_{w \in B_n} w.\]
\end{proof}

Next we use the factorization to obtain the statistical distribution.

\begin{cor}\label{cor:B2}
We have \[ B_n(q,t) = \sum_{w \in B_n} q^{\sor_B(w)}t^{\ell'_B(w)}.\]
\end{cor}

\begin{proof}
This proof is nearly identical to previous arguments. We show that the linear map $\phi: \mathbb{Z}[B_n] \to \mathbb{Z}[q,t]$ with $\phi(w) = q^{\sor_B(w)} t^{\ell'_B(w)}$ obeys \[ \phi(\Phi_{1} \cdots \Phi_{n-1})\phi(\Phi_n) = \phi(\Phi_1 \cdots\Phi_{n-1} \Phi_n).\]

This follows by induction on $n$ and the observation that $\Phi_n$ is all transpositions involving $n$ and hence the products of transpositions generated are indeed reduced.
\end{proof}

We state our desired equidistribution result as follows.

\begin{cor}
The pairs of statistics $(\inv_B, \mb)$ and $(\sor_B, \ell'_B)$ are equidistributed over $B_n$: \[ \sum_{w \in B_n} q^{\inv_B(w)}t^{\mb(w)} = \sum_{w \in B_n} q^{\sor_B(w)}t^{\ell'_B(w)}.\] In particular, $\sor_B$ is Mahonian and $\mb$ is Stirling.
\end{cor}

\section{type D}\label{sec:D}

The type $D_n$ Coxeter group can be defined as the subgroup of $B_n$ ($n\geq 4$) given by elements with an even number of minus signs. The minimal generating set for $D_n$ is \[ S^D := \{ (\bar 1 2) \} \cup \{ (i\, i+1) : 1\leq i \leq n-1\}.\] We denote these simple transpositions by $s_i = (i \, i+1)$ and $s_{\bar 1} = (\bar 1 2)$.

We will use the same notation for transpositions as in the type $B_n$ case: \[ T^D = \{ (i j) : 1\leq i < j \leq n\} \cup \{ (\bar i j) : 1\leq i \leq j \leq n \},\] with $t_{ij} = (ij)$. However, we caution that as opposed to the type B case, the transpositions $t_{\bar i i}$ are \emph{not} reflections in $D_n$. (Recall that a reflection in a Coxeter system $(W,S)$ is any conjugate of an element of $S$, i.e., $t$ is a reflection if and only if $t = wsw^{-1}$ for some $s \in S$ and $w \in W$.) In particular, in $D_n$ we have $t_{\bar 1 1} = 1$ and for $1<i\leq n$ and $w=w_1\cdots w_n$, \[w t_{\bar i i} = \bar{w}_1 w_2 \cdots w_{i-1} \bar{w}_i w_{i+1}\cdots w_n.\] 

The type D sorting index will be defined in terms of a certain factorization by transpositions in $T^D$, but, as not all elements of $T^D$ are reflections, it does not share the same close relationship with reflection length that was seen in types A and B. Thus we consider only the Mahonian statistics in this case.

\subsection{Mahonian statistics}

The length of an element $w \in D_n$, denoted $\ell_D(w)$, is the minimal number of adjacent transpositions from $S^D$ needed to express $w$. The type $D_n$ inversion number is defined as: \[\inv_D(w) = | \{ 1\leq i < j \leq n : w(i) > w(j)\}| + |\{1\leq i < j\leq n : -w(i) > w(j) \}|.\] For example, $\inv_D(\bar 3 24 \bar 5 1) = (1+2+2) + (3+1+1+1) = 11$. We have $\ell_D(w) = \inv_D(w)$ for $w \in D_n$.

Although it is not necessarily a product of reflections, there is nonetheless a unique factorization \[ w = t_{i_1 j_1} \cdots t_{i_k j_k},\] with $1<j_1 < \cdots < j_k$, and the \emph{type $D_n$ sorting index}, $\sor_D(w)$, is defined to be \[ \sor_D(w) = \sum_{s=1}^k j_s - i_s - 2\cdot\chi(i_s < 0).\] The sorting index of type D differs from the type B version in two key ways. First, we use only the transpositions $(i j)$ with $j \geq 2$ (in particular $(\bar 1 1)$ is not used), and second, when $i$ is negative the distance between $i$ and $j$ is $j-i-2$. This is intuitively nice when we think of elements of $D_n$ not as symmetric chains, but as symmetric posets with a ``fork"  in the middle: \[ w_{\bar n} \cdots w_{\bar 2} \begin{array}{c} w_1 \\ w_{\bar 1} \end{array} w_2 \cdots w_n.\] For example, consider the following sorting of the element $w=\bar 3 24 \bar 5 1 \in D_5$. We have \[ \bar 1 \mathbf{5} \bar 4 \bar 2 \begin{array}{c}\bar 3 \\ 3 \end{array} 24 \bar 5 1 \xrightarrow{(\bar 4 5)} \bar 5 1 \bar 4 \bar 2 \begin{array}{c}\bar 3 \\ 3 \end{array} 2 \mathbf{4} \bar 1 5 \xrightarrow{(34)} \bar 5 \bar 4 1 \bar 2 \begin{array}{c}\bar 3 \\ \mathbf{3} \end{array} 2  \bar 1 4 5 \xrightarrow{(\bar 1 3)} \bar 5 \bar 4 \bar 3 \bar 2 \begin{array}{c}1 \\ \bar 1 \end{array} 2 3 4 5 \] and so $\sor_D(w) = (5-(-4)- 2) + (4 -3) + (3-(-1)-2) = 10.$

\subsection{Factorizations of the diagonal sum}

As with ordinary and signed permutations, there are at least two natural ways to generate elements of $D_n$ given the elements of $D_{n-1}$. For the first of these, define $\Psi_1 = 1+s_1 + s_{\bar 1} + s_1s_{\bar 1}$, and for $i \geq 2$, define $\Psi_i = 1 + s_i \Psi_{i-1} + s_i \cdots s_2 s_1 s_{\bar 1} s_2 \cdots s_i.$ We have the following factorization of the diagonal sum.

\begin{prp}\label{prp:prodD}
For $n\geq 4$, we have 
\begin{equation}\label{eq:prodD}
 \Psi_1 \Psi_2 \cdots \Psi_{n-1} = \sum_{w \in D_n} w.
\end{equation}
\end{prp}

\begin{proof}
One can verify (preferably with computer assistance) that $\Psi_1\Psi_2\Psi_3 = \sum_{w\in D_4} w$. For induction, we simply examine the effect that $\Psi_{n-1}$ has on an element $w = w_1\cdots w_{n-1} \in D_{n-1}$. We have:
\begin{align*}
w\cdot\Psi_{n-1} &= w + ws_{n-1} + \cdots + ws_{n-1}\cdots s_{2} + w(s_{n-1}\cdots s_2s_{1})+ w(s_{n-1}\cdots s_{2}s_{\bar 1})w \\
& \quad + w(s_{n-1}\cdots s_2s_1s_{\bar 1}) + \cdots + w(s_{n-1}\cdots s_2s_1s_{\bar 1}s_2\cdots s_{n-1}) \\
&= w_1 \cdots w_{n-1}n + w_1 \cdots w_{n-2}nw_{n-1} + \cdots + w_1n\cdots w_{n-1} + nw_1\cdots w_{n-1} \\
& \quad +\bar n \bar{w}_1 \cdots w_{n-1} + \bar{w}_1\bar n \cdots w_{n-1} + \cdots + \bar{w}_1\cdots w_{n-1}\bar n.
\end{align*}
The result now follows as in the proof of Proposition \ref{prp:prodB}.
\end{proof}

Let $\psi: \mathbb{Z}[D_n] \to \mathbb{Z}[q]$ by $\psi(w) = q^{\inv_D(w)}$, so that $\psi(\Psi_i) = (1+q+\cdots +q^{i-1} + 2q^i + q^{i+1} +\cdots + q^{2i}) = (1+q^{i})[i+1]_q.$ By applying $\psi$ to both sides of \eqref{eq:prodD}, we have the following. (The second equality follows from the fact that $[i]_q(1+q^i) = [2i]_q$.)

\begin{cor}\label{cor:D}
We have \[ D_n(q) := \sum_{w \in D_n} q^{\inv_D(w)} = \prod_{i=1}^{n-1}(1+q^i)[i+1]_q  = [n]_q\cdot\prod_{i=1}^{n-1} [2i]_q.\]
\end{cor}

We omit the proof as it closely follows the related argument for type B.

We now give a factorization analogous to Propositions \ref{prp:prodA2} and \ref{prp:prodB2}. Let $\Phi_j = 1+\sum_{j > i\geq \bar j} t_{ij}$ be as defined in Proposition \ref{prp:prodB2}.

\begin{prp}\label{prp:prodD2}
For $n\geq 4$, we have \begin{equation}\label{eq:prodD2} 
\Phi_2 \Phi_3 \cdots \Phi_n = \sum_{w \in D_n} w.
\end{equation}
\end{prp}

\begin{proof}
One can verify the claim when $n=4$. The general case follows by induction as with the earlier proofs. For consistency, we describe the effect of $\Phi_n$ on an element $w = w_1\cdots w_{n-1}n$. We have:
\begin{align*}
w\cdot\Phi_n &= w + wt_{n-1\,n}  + wt_{n-2\,n} + \cdots + wt_{1 n} + wt_{\bar 1 n}+ wt_{\bar 2 n} + \cdots + wt_{\bar n n} \\
&= w_1 \cdots w_{n-1} n+ w_1 \cdots n w_{n-1} + w_1 \cdots n w_{n-1} w_{n-2} + \cdots + n w_2 \cdots w_{n-1} w_1 \\
& \quad + \bar n w_2 \cdots w_{n-1} \bar{w}_1 + w_1 \bar n \cdots w_{n-1} \bar{w}_2 + \cdots + \bar{w}_1 \cdots w_{n-1} \bar n,
\end{align*} 
from which the result follows.
\end{proof}

Define $\phi: \mathbb{Z}[D_n] \to \mathbb{Z}[q]$ with $\phi(w) = q^{\sor_D(w)}$. By construction, $\phi(\Phi_i) = (1+q+\cdots q^{i-2} + 2q^{i-1} + q^i + \cdots +q^{2i-2}) = (1 + q^{i-1})[i]_q$. Applying $\phi$ to \eqref{eq:prodD2}, we have the following, for which we again omit the proof.

\begin{cor}\label{cor:D2}
For $n\geq 4$, we have
\[ D_n(q) = \sum_{w \in D_n} q^{\sor_D(w)} = \prod_{i=1}^{n-1} (1+q^i)[i+1]_q = [n]_q \cdot \prod_{i=1}^{n-1} [2i]_q.\]
\end{cor}

From Corollaries \ref{cor:D2} and \ref{cor:D} we have shown that $\sor_D$ is Mahonian.

\begin{cor}
The statistics $\inv_D$ and $\sor_D$ are equidistributed over $D_n$: \[ \sum_{w \in D_n} q^{\inv_D(w)} = \sum_{w \in D_n} q^{\sor_D(w)}.\] That is, $\sor_D$ is Mahonian.
\end{cor}

\section{Finite Coxeter groups}\label{sec:W}

Given a Coxeter system $(W,S)$ and an element $w \in W$, we define the \emph{length}, $\ell(w)$,  to be the minimal number of terms needed to express $w$ as a product of elements in the set $S$. The set of all reflections of $(W,S)$ is the set $T := \{ wsw^{-1} : w \in W, s \in S\}$. The \emph{reflection length}, $\ell'(w)$, is the minimal number of terms needed to express $w$ as a product of elements in the set $T$. We have the following classical results on the distribution of these statistics for finite Coxeter groups $W$. See \cite[Sections 3.9 and 3.15]{H}.

\begin{thm}
For a finite Coxeter group $W$,
\begin{itemize}
\item[(a)] (Solomon)
\[ \sum_{w \in W} q^{\ell(w)} = \prod_{i=1}^n [e_i+1]_q,\] and
\item[(b)] (Shepard-Todd)
\[\sum_{w \in W} t^{\ell'(w)} = \prod_{i=1}^n (1+ e_it),\]
\end{itemize}
where the $e_i$ are the \emph{exponents} of $W$.
\end{thm}

\begin{table}
\begin{tabular}{|c||c|}
\hline
$W$ & $e_1,\ldots,e_n$ \\
\hline
$A_n$ & $1,2,3\ldots,n$ \\
$B_n$ & $1,3,5\ldots,2n-3,2n-1$ \\
$D_n$ & $1,3,5\ldots,2n-3,n-1$\\
\hline
\end{tabular}
\caption{Exponents for the Coxeter groups of types $A_n, B_n$, and $D_n$.}
\end{table}

Thus, both the length and reflection length generating functions have a common refinement, defined as follows:
\[ W(q,t) := \prod_{i=1}^n (1 + t[e_i+1]_q - t),\]
and implicitly there are statistics $m$ and $h$ such that:
\begin{align*}
W(q,t) &= \sum_{w \in W} q^{h(w)} t^{\ell'(w)}\\
 &= \sum_{w \in W} q^{\ell(w)} t^{m(w)}.
\end{align*}

For $W=A_n$ and $W=B_n$ we have given explicit descriptions for both $h$\footnote{We use $h$ to suggest \emph{height}, as used in the theory of root systems. Each reflection $t$ (in a Weyl group) corresponds to a unique positive root $\beta$. As we have defined things, in both $A_n$ and $B_n$, $h(\beta) = \sor(t)$, and if $w = t_1\cdots t_k$ is the factorization produced by sorting, then $h(\beta_1)+\cdots + h(\beta_k) = \sor(w)$.} and $m$. It would be nice to have a type-independent description of these statistics. Notice that the $i$th factors Propositions \ref{prp:prodA}, \ref{prp:prodA2}, \ref{prp:prodB}, and \ref{prp:prodB2} each have $e_i+1$ terms. This is not the case for Propositions \ref{prp:prodD} and \ref{prp:prodD2}. In fact, one can show for $D_4$ that such a factorization (at least one that is well-behaved with respect to length or reflection length) does not exist.

\section{Final remarks}\label{sec:rem}

Upon completion of this work, the author was made aware of independent work of Wilson \cite[Sec. 2.2]{W}, in which the type A version of the sorting index is described. Wilson has explored the sorting index further in a recent paper \cite{W2}. This paper establishes in particular that the sorting index is not trivially equivalent to many well-known Mahonian statistics. Nonetheless, work of Galovich and White \cite{GW} gives a recursive method for constructing a bijection $f: S_n  \to S_n$ such that $\sor(w)$ to $\inv(f(w))$. This follows because the sorting index is what Galovich and White call a ``splittable Mahonian" statistic.

Another connection to the type A sorting index appears in recent work of Foata and Han \cite{FH}. In particular, their ``B-code" $(b_1,\cdots, b_n)$ of $w \in S_n$ is given by  $b_{j_s} = j_s -i_s$ for each $j_s$ that appears appears in the factorization $w = t_{i_1 j_1}\cdots t_{i_k j_k}$ of Section \ref{sec:mahA} ($b_j = 0$ otherwise). Moreover, the final paragraph of  \cite[Section 6]{FH} describes a related statistic ``env" which gives rise to an equidistribution result similar to Corollary \ref{cor:Adis}.

It is perhaps unsurprising that several researchers have independently discovered connections with the type A sorting index, as ``straight selection sort" is such a natural operation.

The author would like to thank Dennis White and Mark Wilson for helpful comments and Einar Steingr\'imsson for pointing out the connection with Wilson's work.

\end{document}